\documentclass[12pt,twoside]{amsart}
\usepackage[all]{xy}
\usepackage{amssymb,amsmath,amsthm, amscd, enumerate, mathrsfs}
\usepackage{graphicx, hhline}
\title[Fujita-type freeness]{Fujita-type 
freeness for quasi-log canonical curves and surfaces}
\author{Osamu Fujino and Haidong Liu} 
\date{2018/12/11, version 0.16}
\address{Department of Mathematics, Graduate School of Science, 
Osaka University, Toyonaka, Osaka 560-0043, Japan}
\email{fujino@math.sci.osaka-u.ac.jp}
\address{Department of Mathematics, Graduate School of Science, 
Kyoto University, Kyoto 606-8502, Japan}
\email{liu.dong.82u@st.kyoto-u.ac.jp}
\DeclareMathOperator{\Supp}{Supp}
\DeclareMathOperator{\Exc}{Exc}
\DeclareMathOperator{\mult}{mult}
\DeclareMathOperator{\Nqlc}{Nqlc}
\DeclareMathOperator{\Nqklt}{Nqklt}
\DeclareMathOperator{\Nklt}{Nklt}
\newtheorem{thm}{Theorem}[section]
\newtheorem{lem}[thm]{Lemma}

\newtheorem{conj}[thm]{Conjecture}
\newtheorem{cor}[thm]{Corollary}

\theoremstyle{definition}
\newtheorem{ex}[thm]{Example}
\newtheorem{defn}[thm]{Definition}

\newtheorem*{ack}{Acknowledgments}       
\newtheorem{step}{Step}
\newtheorem{case}{Case}
\makeatletter
    
    \@addtoreset{equation}{section}
\makeatother

\begin{document}

\maketitle 

\begin{abstract}
We prove Fujita-type basepoint-freeness for projective 
quasi-log canonical curves and surfaces. 
\end{abstract}

\section{Introduction}\label{a-sec1}

Fujita's freeness conjecture is 
very famous and is still open for 
higher-dimensional varieties. 
Now we know that 
it holds true in dimension $\leq 5$ (for the 
details, see \cite{ye-zhu} and the references therein). 

\begin{conj}[Fujita's freeness conjecture]\label{a-conj1.1} 
Let $X$ be a smooth projective variety of dimension $n$.
Let $L$ be an ample Cartier divisor.
Then the complete linear system 
$|K_{X}+(n+1)L|$ is basepoint-free. 
\end{conj}

In this paper, we treat a generalization of 
Fujita's freeness conjecture for highly singular varieties. 
More precisely, we are mainly interested in quasi-log canonical pairs. 
A quasi-log canonical pair may be reducible and is not necessarily 
equidimensional. 
The union of some log canonical centers of a given log canonical 
pair is a typical example of quasi-log canonical pairs. 
We think that it is worth formulating and studying 
various conjectures for quasi-log canonical pairs 
in order to solve the original 
conjecture by some inductive arguments on the dimension.  

\begin{conj}[Fujita-type freeness for quasi-log canonical 
pairs]\label{a-conj1.2}
Let $[X, \omega]$ be a projective quasi-log canonical 
pair of dimension $n$. 
Let $M$ be a Cartier divisor on $X$. We put $N=M-\omega$. 
Assume that $N^{\dim X_i}\cdot X_i>
(\dim X_i)^{\dim X_i}$ for every positive-dimensional 
irreducible component 
$X_i$ of $X$. 
For every positive-dimensional 
subvariety $Z$ which is not an irreducible 
component of $X$, 
we put 
$$n_Z=\min_i\{ \dim X_i \, |\, 
{\text{$X_i$ is an irreducible component of $X$ with 
$Z\subset X_i$}}\}$$ and assume 
that $N^{\dim Z}\cdot Z\geq n_Z^{\dim Z}$. 
Then the complete linear system $|M|$ is basepoint-free. 
\end{conj}

If $N^{\dim X_i}\cdot X_i>\left(\frac{1}{2}n(n+1)\right)^{\dim X_i}$ and 
$N^{\dim Z}\cdot Z>\left(\frac{1}{2}n(n+1)\right)^{\dim Z}$ 
hold in Conjecture 
\ref{a-conj1.2}, then we have already known that 
the complete linear system 
$|M|$ is basepoint-free by the second author's 
theorem (see \cite[Theorem 1.1]{liu} for the precise statement). 
It is a generalization of Angehrn--Siu's theorem (see \cite{angehrn-siu}). 
When $\dim X=1$, 
we can easily check that Conjecture \ref{a-conj1.2} holds true. 

\begin{thm}[Theorem \ref{a-thm3.1}]\label{a-thm1.3}
Conjecture \ref{a-conj1.2} holds 
true for $n=1$. 
\end{thm}

The main technical result of this paper is the following theorem.  

\begin{thm}[Theorem \ref{a-thm3.2}]\label{a-thm1.4}
Let $[X, \omega]$ be a quasi-log canonical pair such that $X$ is a normal 
projective irreducible surface. 
Let $M$ be a Cartier divisor on $X$. 
We put $N=M-\omega$. 
We assume that $N^2>4$ and $N\cdot C\geq 2$ for every 
curve $C$ on $X$. 
Let $P$ be any closed point of $X$ that is not included in 
$\Nqklt(X, \omega)$, 
the union of all qlc centers of $[X, \omega]$. 
Then there exists $s\in H^0(X, \mathcal I_{\Nqklt(X, \omega)}
\otimes \mathcal O_X(M))$ such that 
$s(P)\ne 0$, where $\mathcal I_{\Nqklt(X, \omega)}$ is 
the defining ideal sheaf of $\Nqklt(X, \omega)$ on $X$. 
\end{thm}

The proof of Theorem \ref{a-thm1.4} in Section \ref{a-sec3} 
heavily depends on the 
first author's new result obtained in \cite{fujino-slc-trivial} (see 
Theorem \ref{u-thm2.12} below), 
which comes from the theory of variations of mixed Hodge structure on 
cohomology with compact support. 
By combining Theorems \ref{a-thm1.3} and 
\ref{a-thm1.4} with our result on the normalization of 
quasi-log canonical pairs (see Theorem \ref{u-thm2.11} below), 
we prove Conjecture \ref{a-conj1.2} for $n=2$ in full 
generality.  

\begin{cor}[Corollary \ref{a-cor3.3}]\label{a-cor1.5}
Conjecture \ref{a-conj1.2} holds true for $n=2$. 
\end{cor}

We note that we can recover the 
main theorem of \cite{fujino-slc-surface} 
by combining Theorem \ref{a-thm1.3} and Corollary \ref{a-cor1.5} 
with the main result of \cite{fujino-funda-slc}.  

\begin{cor}[{\cite[Theorem 1.3]{fujino-slc-surface}}]\label{a-cor1.6} 
Let $(X, \Delta)$ be a projective 
semi-log canonical 
pair of dimension $n$. 
Let $M$ be a Cartier divisor on $X$. 
We put $N=M-(K_X+\Delta)$. 
Assume that $N^n\cdot X_i>n^n$ for every irreducible 
component $X_i$ of $X$ and 
that $N^k\cdot Z\geq n^k$ for every subvariety $Z$ with $0<\dim Z=k<n$. 
We further assume that 
$n=1$ or $2$. 
Then the complete linear system $|M|$ is basepoint-free.  
\end{cor}

Let us quickly explain 
our strategy to prove Conjecture \ref{a-conj1.2}. 
From now on, we will use the same notation as in Conjecture \ref{a-conj1.2}. 
We take an arbitrary closed point $P$ of $X$. 
Then it is sufficient to find $s\in H^0(X, \mathcal O_X(M))$ with 
$s(P)\ne 0$. 
Let $X_i$ be an irreducible component of $X$ such that 
$P\in X_i$. 
By adjunction (see Theorem \ref{u-thm2.8} (i)), 
$[X_i, \omega|_{X_i}]$ is a quasi-log canonical pair. 
By the vanishing theorem (see Theorem \ref{u-thm2.8} (ii)), 
the natural restriction map 
$H^0(X, \mathcal O_X(M))\to H^0(X_i, \mathcal O_{X_i}(M))$ is 
surjective. Therefore, by replacing $X$ with $X_i$, 
we may assume that $X$ is irreducible. 
By adjunction again, 
$[\Nqklt(X, \omega), \omega|_{\Nqklt(X, \omega)}]$ is a quasi-log 
canonial pair. 
By the vanishing theorem, 
the natural restriction map $H^0(X, \mathcal O_X(M))
\to H^0(\Nqklt(X, \omega), \mathcal O_{\Nqklt(X, \omega)}(M))$ 
is surjective. Therefore, if $P\in \Nqklt(X, \omega)$, 
then we can use induction on the dimension. 
Thus we may further assume that 
$P\not\in\Nqklt(X, \omega)$. 
In this situation, we know that $X$ is normal at $P$. 
Let $\nu: \widetilde X\to X$ be the normalization. 
Then, by Theorem \ref{u-thm2.11}, $[\widetilde X, 
\nu^*\omega]$ is a quasi-log canonical pair with 
$\nu_*\mathcal I_{\Nqklt(\widetilde X, \nu^*\omega)}=\mathcal I_{\Nqklt(X, \omega)}$. 
Therefore, it is sufficient to find $\widetilde s\in 
H^0(\widetilde X, \mathcal I_{\Nqklt(\widetilde X, \nu^*\omega)}\otimes 
\mathcal O_{\widetilde X} (\nu^*M))$ with $\widetilde s(\widetilde P)\ne 0$, 
where $\widetilde P=\nu^{-1}(P)$. 
By replacing $X$ with $\widetilde X$, we may assume that 
$X$ is a normal irreducible variety. 
By using Theorem \ref{u-thm2.12}, we can take a 
boundary $\mathbb R$-divisor $\Delta$, that is, 
an effective $\mathbb R$-divisor $\Delta$ with $\Delta=\Delta^{\leq 1}$, on 
$X$ such that $K_X+\Delta\sim _{\mathbb R} \omega+\varepsilon N$ for 
$0<\varepsilon \ll 1$ and $\mathcal J(X, \Delta)=\mathcal I_{\Nqklt(X, \omega)}$. 
Note that $\mathcal J(X, \Delta)$ is the multiplier ideal sheaf of $(X, \Delta)$. 
Since $\mathcal J(X, \Delta)=\mathcal I_{\Nqklt(X, \omega)}$, $(X, \Delta)$ is 
klt in a neighborhood of $P$. 
Anyway, it is sufficient to find $s\in H^0(X, \mathcal J(X, \Delta)\otimes 
\mathcal O_X(M))$ with $s(P)\ne 0$. 
In this paper, we will carry out the above strategy in dimension two. 

\begin{ack}
The first author was partially supported by JSPS KAKENHI Grant 
Numbers JP16H03925, JP16H06337. 
The authors would like to thank Kenta Hashizume and 
Professor Wenfei Liu for discussions. 
They also thank the referee for valuable comments. 
\end{ack}

We will work over $\mathbb C$, the 
complex number field, throughout this paper. A scheme means
a separated scheme of finite 
type over $\mathbb C$. A variety means a reduced scheme, that is, a
reduced separated scheme of finite 
type over $\mathbb C$. 
We sometimes assume that a variety 
is irreducible 
without mentioning it explicitly if there is no risk of 
confusion. We will freely use the standard notation of 
the minimal model program and the 
theory of quasi-log schemes as in  \cite{fujino-funda} and  \cite{fujino-foundations}.
For the details of semi-log canonical pairs, see \cite{fujino-funda-slc}. 

\section{Preliminaries}\label{a-sec2}

In this section, we collect some basic definitions and 
explain some results on quasi-log schemes. 

\begin{defn}[$\mathbb R$-divisors]\label{u-def2.1}
Let $X$ be an equidimensional variety, 
which is not necessarily regular in codimension one. 
Let $D$ be an $\mathbb R$-divisor, 
that is, $D$ is a finite formal sum $\sum _i d_i D_i$, where 
$D_i$ is an irreducible reduced closed subscheme of $X$ of 
pure codimension one and $d_i$ is a real number for every $i$ 
such that $D_i\ne D_j$ for $i\ne j$. We put 
\begin{equation*}
D^{<1} =\sum _{d_i<1}d_iD_i, \quad 
D^{\leq 1}=\sum _{d_i\leq 1} d_i D_i, \quad 
D^{> 1}=\sum _{d_i>1} d_i D_i, \quad 
\text{and} \quad 
D^{=1}=\sum _{d_i=1}D_i. 
\end{equation*}
We also put 
$$
\lceil D\rceil =\sum _i \lceil d_i \rceil D_i \quad \text{and} 
\quad 
\lfloor D\rfloor=-\lceil -D\rceil, 
$$
where $\lceil d_i\rceil$ is the integer defined by $d_i\leq 
\lceil d_i\rceil <d_i+1$. 
When $D=D^{\leq 1}$ holds, 
we usually say that 
$D$ is a {\em{subboundary}} $\mathbb R$-divisor. 

Let $B_1$ and $B_2$ be $\mathbb R$-Cartier divisors on $X$. 
Then $B_1\sim _{\mathbb R} B_2$ means that 
$B_1$ is $\mathbb R$-linearly equivalent 
to $B_2$. 
\end{defn}

Let us quickly recall singularities of pairs for the reader's convenience. 
We recommend the reader to see \cite[Section 2.3]{fujino-foundations} 
for the details.  

\begin{defn}[Singularities of pairs]\label{u-def2.2} 
Let $X$ be a normal variety and let $\Delta$ be 
an $\mathbb R$-divisor on $X$ such that 
$K_X+\Delta$ is $\mathbb R$-Cartier. 
Let $f:Y\to X$ be a projective 
birational morphism from a smooth variety $Y$. 
Then we can write 
$$
K_Y=f^*(K_X+\Delta)+\sum _E a(E, X, \Delta)E, 
$$ 
where $a(E, X, \Delta)\in \mathbb R$ and $E$ is a prime 
divisor on $Y$. 
By taking $f:Y\to X$ suitably, 
we can define $a(E, X, \Delta)$ for any prime 
divisor $E$ {\em{over}} $X$ and call it 
the {\em{discrepancy}} of $E$ with respect to $(X, \Delta)$. 
If $a(E, X, \Delta)>-1$ (resp.~$a(E, X, \Delta)\geq -1$) holds 
for any prime divisor $E$ over $X$, then 
we say that $(X, \Delta)$ is {\em{sub klt}} (resp.~{\em{sub 
log canonical}}). 
If $(X, \Delta)$ is sub klt (resp.~sub log canonical) 
and $\Delta$ is effective, then we say that 
$(X, \Delta)$ is {\em{klt}} (resp.~{\em{log canonical}}). 
If $(X, \Delta)$ is log canonical 
and $a(E, X, \Delta)>-1$ for any prime divisor 
$E$ that is exceptional over $X$, 
then we say that $(X, \Delta)$ is {\em{plt}}. 

If there exist a projective birational morphism 
$f:Y\to X$ from a smooth 
variety $Y$ and a prime divisor $E$ on $Y$ such that 
$a(E, X, \Delta)=-1$ and $(X, \Delta)$ is log canonical 
in a neighborhood of the generic point of $f(E)$, 
then $f(E)$ is called a {\em{log canonical center}} of 
$(X, \Delta)$.  
\end{defn}

\begin{defn}[Multiplier ideal sheaves]\label{u-def2.3}
Let $X$ be a normal variety and let $\Delta$ 
be an effective $\mathbb R$-divisor 
on $X$ such that $K_X+\Delta$ is $\mathbb R$-Cartier. 
Let $f:Y\to X$ be a projective 
birational morphism 
from a smooth variety 
such that 
$$
K_Y+\Delta_Y=f^*(K_X+\Delta)
$$ 
and $\Supp \Delta_Y$ is a simple normal crossing 
divisor on $Y$. 
We put 
$$
\mathcal J(X, \Delta)=f_*\mathcal O_Y(-\lfloor \Delta_Y\rfloor)
$$ 
and call it the {\em{multiplier ideal 
sheaf}} of $(X, \Delta)$. 
We can easily check that 
$\mathcal J(X, \Delta)$ is a well-defined 
ideal sheaf on $X$. 
The closed subscheme defined by 
$\mathcal J(X, \Delta)$ is denoted by 
$\Nklt(X, \Delta)$. 
\end{defn}

The notion of {\em{globally embedded simple normal crossing 
pairs}} plays a crucial role in the theory of quasi-log schemes 
described in \cite[Chapter 6]{fujino-foundations}. 

\begin{defn}[Globally embedded simple normal 
crossing pairs]\label{u-def2.4} 
Let $Y$ be a simple normal crossing 
divisor on a smooth variety $M$ and let $B$ be 
an $\mathbb R$-divisor 
on $M$ such that 
$Y$ and $B$ have no common irreducible components and 
that the support of $Y+B$ is a simple normal crossing divisor on $M$. In this 
situation, $(Y, B_Y)$, where $B_Y:=B|_Y$, 
is called a {\em{globally embedded simple 
normal crossing pair}}. 
A {\em{stratum}} of $(Y, B_Y)$ means a log canonical 
center of $(M, Y+B)$ included in $Y$.  
\end{defn}

Let us recall the notion of {\em{quasi-log schemes}}, 
which was first introduced by Florin Ambro (see \cite{ambro}). 
The following definition is slightly different from 
the original one. 
For the details, see \cite[Appendix A]{fujino-pull-back}. 
In this paper, 
we will use the framework of quasi-log schemes 
established in 
\cite[Chapter 6]{fujino-foundations}. 

\begin{defn}[Quasi-log schemes]\label{u-def2.5} 
A {\em{quasi-log scheme}} is a scheme $X$ endowed with 
an $\mathbb R$-Cartier divisor (or $\mathbb R$-line bundle) $\omega$ on $X$, 
a closed subscheme $\Nqlc(X, \omega)\subsetneq X$,  
and a finite collection $\{C\}$ of reduced and irreducible subschemes 
of $X$ such that there exists a proper morphism 
$f:(Y, B_Y)\to X$ from a globally embedded simple normal 
crossing pair $(Y, B_Y)$ satisfying 
the following properties: 
\begin{itemize}
\item[(1)] $f^*\omega\sim_{\mathbb R} K_Y+B_Y$. 
\item[(2)] The natural map 
$\mathcal O_X\to f_*\mathcal O_Y(\lceil -(B^{<1}_Y)\rceil)$ induces 
an isomorphism 
$$
\mathcal I_{\Nqlc(X, \omega)}\overset{\sim}\longrightarrow
f_*\mathcal O_Y(\lceil -(B^{<1}_Y)\rceil-\lfloor B^{>1}_Y\rfloor), 
$$ 
where $\mathcal I_{\Nqlc(X, \omega)}$ is the defining ideal sheaf of 
$\Nqlc(X, \omega)$. 
\item[(3)] The collection of subvarieties $\{C\}$ 
coincides with the images of $(Y, B_Y)$-strata that are 
not included in $\Nqlc(X, \omega)$.  
\end{itemize}
We simply write $[X, \omega]$ to denote 
the above data 
$$\left(X, \omega, f:(Y, B_Y)\to X\right)
$$ 
if there is no risk of confusion. 
We note that the subvarieties $C$ are called the {\em{qlc strata}} 
of $\left(X, \omega, f:(Y, B_Y)\to X\right)$ or simply of $[X, \omega]$. 
If $C$ is a qlc stratum of $[X, \omega]$ but is not an irreducible 
component of $X$, then $C$ is called a {\em{qlc center}} 
of $[X, \omega]$. 
The union of all qlc centers of $[X, \omega]$ is denoted by 
$\Nqklt(X, \omega)$. 
\end{defn}

If $B_Y$ is a subboundary $\mathbb R$-divisor, 
then $[X, \omega]$ in Definition \ref{u-def2.5} is 
called a {\em{quasi-log canonical pair}}. 

\begin{defn}[Quasi-log canonical pairs]\label{u-def2.6}
Let $(X, \omega, f:(Y, B_Y)\to X)$ be a quasi-log scheme as 
in Definition \ref{u-def2.5}. 
We say that $(X, \omega, f:(Y, B_Y)\to X)$ or simply $[X, \omega]$ 
is a {\em{quasi-log canonical pair}} ({\em{qlc pair}}, for short) 
if $\Nqlc(X, \omega)=\emptyset$. 
Note that the condition $\Nqlc(X, \omega)=\emptyset$ is equivalent 
to $B^{>1}_Y=0$, that is, $B_Y=B^{\leq 1}_Y$. 
\end{defn}

The following example is very important. 
Precisely speaking, the notion of quasi-log schemes was 
originally introduced by Florin Ambro (see \cite{ambro}) 
in order to establish the cone and contraction theorem 
for {\em{generalized log varieties}}. 
Note that a generalized log variety $(X, \Delta)$ means 
that $X$ is a normal variety and $\Delta$ is an effective 
$\mathbb R$-divisor on $X$ such that 
$K_X+\Delta$ is $\mathbb R$-Cartier as in Example \ref{u-ex2.7} below. 

\begin{ex}\label{u-ex2.7}
Let $X$ be a normal irreducible variety 
and let $\Delta$ be an effective $\mathbb R$-divisor 
on $X$ such that 
$K_X+\Delta$ is $\mathbb R$-Cartier. 
Let $f:Y\to X$ be a projective 
birational morphism from a smooth variety $Y$. 
We define $\Delta_Y$ by $$K_Y+\Delta_Y=f^*(K_X+\Delta). $$
We may assume that $\Supp \Delta_Y$ is a simple 
normal crossing divisor 
on $Y$ by taking $f:Y\to X$ suitably. 
We put $M=Y\times \mathbb C$ and 
consider $Y\simeq Y\times \{0\}\hookrightarrow 
Y\times \mathbb C=M$. 
Then we can see $(Y, \Delta_Y)$ as a globally embedded 
simple normal crossing pair. 
We put $\omega:=K_X+\Delta$ and 
$$\mathcal I_{\Nqlc(X, \omega)}:=f_*\mathcal O_Y(\lceil 
-(\Delta^{<1}_Y)\rceil -\lfloor \Delta^{>1}_Y\rfloor)\subset 
\mathcal O_X. $$ 
Then $\left(X, \omega, f: (Y, \Delta_Y)\to X\right)$ is a quasi-log 
scheme. 
In this case, $C$ is a qlc center of $[X, \omega]$ 
if and only if $C$ is a log canonical center of $(X, \Delta)$. 
If $C$ is a qlc stratum but is not a qlc center of $[X, \omega]$, 
then $C$ is nothing but $X$. 
\end{ex}

One of the most important results in the theory of quasi-log schemes 
is the following theorem. 

\begin{thm}\label{u-thm2.8} 
Let $[X,\omega]$ be a 
quasi-log scheme and let $X'$ be the union of $\Nqlc(X, \omega)$ with a 
$($possibly empty$)$ union 
of some qlc strata of $[X,\omega]$. 
Then we have the following properties.
\begin{itemize}
\item[(i)] {\em{(Adjunction)}}.~Assume 
that $X'\neq \Nqlc(X, \omega)$. 
Then $[X', \omega']$ is a quasi-log scheme with $\omega'=\omega \vert_{X'}$ 
and $\Nqlc(X', \omega')=\Nqlc(X, \omega)$. Moreover, the qlc 
strata of $[X',\omega']$ are exactly the qlc strata 
of $[X,\omega]$ that are included in $X'$.
\item[(ii)] {\em{(Vanishing theorem)}}.~Assume 
that $ \pi: X \rightarrow S$ is a proper 
morphism between schemes. 
Let $L$ be a Cartier divisor on $X$ such that 
$L-\omega$ is nef and log big over $S$ 
with respect to $[X,\omega]$, that is, 
$L-\omega$ is $\pi$-nef and 
$(L-\omega)|_C$ is $\pi$-big for every qlc stratum $C$ of $[X, \omega]$. 
Then $R^{i}\pi_{*}(\mathcal {I}_{X'}\otimes \mathcal{O}_{X}(L))=0$ 
for every $i>0$, where $\mathcal{I}_{X'}$ is the 
defining ideal sheaf of $X'$ on $X$.
\end{itemize}
\end{thm}

For the proof of Theorem \ref{u-thm2.8}, see, 
for example, \cite[Theorem 6.3.5]
{fujino-foundations}. 
We note that we generalized Koll\'ar's torsion-free and 
vanishing theorems in \cite[Chapter 5]{fujino-foundations} 
by using the theory of mixed Hodge structures on cohomology with 
compact support 
in order to establish Theorem \ref{u-thm2.8}. 

\medskip

Let us quickly recall the definition of semi-log canonical 
pairs for the reader's convenience. 

\begin{defn}[Semi-log canonical pairs]\label{u-def2.9}
Let $X$ be an equidimensional 
variety that is normal crossing in codimension one 
and satisfies Serre's $S_2$ condition and 
let $\Delta$ be an effective $\mathbb R$-divisor 
on $X$ such that 
the singular locus of $X$ contains no irreducible components 
of $\Supp \Delta$. 
Assume that $K_X+\Delta$ is $\mathbb R$-Cartier. 
Let $\nu:\widetilde X\to X$ be the normalization. 
We put $K_{\widetilde X}+\Delta_{\widetilde X}=\nu^*(K_X+\Delta)$, 
that is, $\Delta_{\widetilde X}$ is the union of 
the inverse images of $\Delta$ and the conductor of 
$X$. 
If $(\widetilde X, \Delta_{\widetilde X})$ is log canonical, 
then $(X, \Delta)$ is called a {\em{semi-log 
canonical pair}}. 
\end{defn}

The theory of quasi-log schemes plays an important role 
for the study of semi-log canonical 
pairs by the following theorem:~Theorem \ref{u-thm2.10}. 
For the precise statement and some related results, 
see \cite{fujino-funda-slc}. 

\begin{thm}[{\cite[Theorem 1.2]{fujino-funda-slc}}]\label{u-thm2.10}
Let $(X, \Delta)$ be a quasi-projective 
semi-log canonical pair. 
Then $[X, K_X+\Delta]$ is a quasi-log canonical pair. 
\end{thm}

For the proof of Corollary \ref{a-cor1.5}, we will 
use Theorem \ref{u-thm2.11} below. 

\begin{thm}[{\cite[Theorem 1.1]{fujino-haidong}}]\label{u-thm2.11} 
Let $[X, \omega]$ be a quasi-log canonical pair such that 
$X$ is irreducible. 
Let $\nu: \widetilde{X}\to X$ be the normalization. 
Then $[\widetilde{X}, \nu^*\omega]$ naturally becomes a 
quasi-log canonical pair with 
the following properties: 
\begin{itemize}
\item[(i)] if $C$ is a qlc center of $[\widetilde{X}, \nu^*\omega]$, 
then $\nu(C)$ is a qlc center of $[X, \omega]$, 
and 
\item[(ii)] $\Nqklt(\widetilde{X}, \nu^*\omega)=\nu^{-1}(\Nqklt(X, \omega))$. 
More precisely, the equality 
$$
\nu_*\mathcal I_{\Nqklt(\widetilde{X}, \nu^*\omega)}=
\mathcal I_{\Nqklt(X, \omega)}
$$ 
holds, where 
$\mathcal I_{\Nqklt(X, \omega)}$ 
and $\mathcal I_{\Nqklt(\widetilde{X}, \nu^*\omega)}$ are the defining ideal sheaves of 
$\Nqklt(X, \omega)$ and $\Nqklt(\widetilde{X}, \nu^*\omega)$ respectively. 
\end{itemize}  
\end{thm}

The following theorem is a special case of \cite[Theorem 1.5]
{fujino-slc-trivial}. It is a deep result 
based on the theory of variations of mixed Hodge structure 
on cohomology with compact support. 

\begin{thm}[{\cite[Theorem 1.5]{fujino-slc-trivial}}]\label{u-thm2.12} 
Let $[X, \omega]$ be a quasi-log canonical pair such that 
$X$ is a normal projective irreducible variety. 
Then there exists a projective 
birational morphism 
$p:X'\to X$ from a smooth projective 
variety $X'$ such that 
$$
K_{X'}+B_{X'}+M_{X'}=p^*\omega, 
$$ 
where $B_{X'}$ is a subboundary $\mathbb R$-divisor, that is, 
$B_{X'}=B^{\leq 1}_{X'}$, 
such that $\Supp B_{X'}$ is a simple normal crossing 
divisor and that $p_*B_{X'}$ is effective,  
and $M_{X'}$ is a nef $\mathbb R$-divisor on $X'$. 
Furthermore, we can make $B_{X'}$ satisfy 
$p(B^{=1}_{X'})=\Nqklt(X, \omega)$.  
\end{thm}

We close this section with an easy lemma, which is 
essentially contained in \cite[Chapter 6]{fujino-foundations}. 

\begin{lem}\label{u-lem2.13} 
Let $[X, \omega]$ be a quasi-log canonical pair such that $X$ is an 
irreducible curve. 
Let $P$ be a smooth point of $X$ such that 
$P$ is not a qlc center of $[X, \omega]$. 
Then we can consider a natural quasi-log structure 
on $[X, \omega+tP]$ induced 
from $[X, \omega]$ for every $t\geq 0$. 
We put 
$$
c=\max \{t\geq 0\, |\, \text{$[X, \omega+tP]$ is quasi-log canonical}\}. 
$$
Then $0<c\leq 1$ holds. 
\end{lem}

\begin{proof} 
Since $[X, \omega]$ is a qlc pair, 
we can take a projective 
surjective morphism 
$f:(Y, B_Y)\to X$ from a globally embedded simple normal crossing 
pair $(Y, B_Y)$ such that 
$B_Y$ is a subboundary $\mathbb R$-divisor 
on $Y$ and that the natural 
map 
$\mathcal O_X\to f_*\mathcal O_Y(\lceil -(B^{<1}_Y)\rceil)$ 
is an isomorphism. 
By taking some blow-ups, 
we may further assume that 
$(Y, \Supp B_Y+\Supp f^*P)$ is a globally 
embedded simple normal crossing pair. 
Then it is easy to see that 
$$\left(X, \omega+tP, f: (Y, B_Y+tf^*P)\to X
\right)
$$ 
is a quasi-log scheme for every $t\geq 0$. 
We assume that $c>1$. 
Then $\mult _S (B_Y+f^*P)<1$ for any irreducible component $S$ of $\Supp 
f^*P$. 
Therefore, $f^*P\leq \lceil -(B^{<1}_Y)\rceil$ holds. 
Thus we have 
$\mathcal O_X\subsetneq \mathcal O_X(P)\subset f_*\mathcal 
O_Y(\lceil -(B^{<1}_Y)\rceil)$. 
This is a contradiction. 
This means that $c\leq 1$ holds. 
By definition, we can easily see that $0<c$ holds.  
\end{proof}

\section{Proof}\label{a-sec3}

In this section, we prove the results in Section \ref{a-sec1}, 
that is, Theorems \ref{a-thm1.3}, \ref{a-thm1.4}, 
Corollaries \ref{a-cor1.5}, and \ref{a-cor1.6}. 

\medskip 

First, we prove Theorem \ref{a-thm1.3}, 
that is, we prove Conjecture \ref{a-conj1.2} when $\dim X=1$. 

\begin{thm}[Theorem \ref{a-thm1.3}]\label{a-thm3.1}
Let $[X,\omega]$ be a
projective quasi-log canonical pair of dimension one.
Let $M$ be a Cartier divisor 
on $X$. We put $N=M-\omega$.  
Assume that $N \cdot X_i> 1$ for 
every one-dimensional irreducible component $ X_i$ of $X$.
Then the complete linear system $|M|$ is 
basepoint-free. 
\end{thm}

\begin{proof}
Let $P$ be an arbitrary closed point of $X$. 
If $P$ is a qlc center of $[X, \omega]$, then 
$H^1(X, \mathcal I_P\otimes \mathcal O_X(M))=0$ 
by Theorem \ref{u-thm2.8} (ii), 
where $\mathcal I_P$ is the defining ideal 
sheaf of $P$ on $X$. 
Therefore, the natural restriction map 
\begin{equation*}
H^0(X, \mathcal O_X(M))\to \mathcal O_X(M)\otimes 
\mathbb C(P)
\end{equation*} 
is surjective. 
Thus, the complete linear system 
$|M|$ is basepoint-free in a 
neighborhood 
of $P$. 
From now on, 
we assume that $P$ is not a qlc center of $[X, \omega]$. 
Let $X_i$ be the unique irreducible 
component of $X$ containing $P$. 
By Theorem \ref{u-thm2.8} (ii), $H^1(X, \mathcal I_{X_i}
\otimes \mathcal O_X(M))=0$, 
where $\mathcal I_{X_i}$ is the defining ideal sheaf of $X_i$ on $X$. 
We note that $X_i$ is a qlc stratum of $[X, \omega]$. 
Thus, the restriction map 
$$
H^0(X, \mathcal O_X(M))\to H^0(X_i, \mathcal O_{X_i}(M)) 
$$ 
is surjective. 
Therefore, by replacing $X$ with 
$X_i$, 
we may assume that $X$ is irreducible. 
By Lemma \ref{u-lem2.13}, we can take $c\in \mathbb R$ such 
that $0<c\leq 1$ and that $P$ is a 
qlc center of $[X, \omega+cP]$. 
Since $\deg (M-(\omega+cP))>1-c\geq 0$, 
we have $H^1(X, \mathcal I_P\otimes \mathcal O_X(M))=0$ 
by Theorem \ref{u-thm2.8} (ii). 
Therefore, by the same argument as above, 
$|M|$ is basepoint-free in a neighborhood 
of $P$. 
Thus we obtain that the complete linear system 
$|M|$ is basepoint-free. 
\end{proof}

Next, we prove Theorem \ref{a-thm1.4}, which is the main 
technical result of this paper. 

\begin{thm}[Theorem \ref{a-thm1.4}]\label{a-thm3.2}
Let $[X, \omega]$ be a quasi-log canonical pair such that $X$ is a normal 
projective irreducible surface. 
Let $M$ be a Cartier divisor on $X$. 
We put $N=M-\omega$. 
We assume that $N^2>4$ and $N\cdot C\geq 2$ for every 
curve $C$ on $X$. 
Let $P$ be any closed point of $X$ that is not included in 
$\Nqklt(X, \omega)$. 
Then there exists $s\in H^0(X, \mathcal I_{\Nqklt(X, \omega)}
\otimes \mathcal O_X(M))$ such that 
$s(P)\ne 0$. 
\end{thm}

\begin{proof}
By assumption and Nakai's ampleness criterion for $\mathbb R$-divisors 
(see \cite{campana-peternell}), $N$ is ample. 
In Step \ref{a-step1}, we will prove Theorem \ref{a-thm3.2} 
under the extra assumption that $P$ is a smooth point of $X$. 
In Step \ref{a-step2}, we will treat the case where $P$ is a 
singular point of $X$.

\begin{step}\label{a-step1}
In this step, we assume that $P$ is a smooth point of $X$. 
Since $N^2>4$, 
we can take an effective $\mathbb R$-divisor $B$ on $X$ 
such that $B\sim _{\mathbb R} N$ with $\mult_P B=2+\alpha>2$. 
By Theorem \ref{u-thm2.12}, 
there exists a projective 
birational morphism 
$p:X'\to X$ from a smooth 
projective surface $X'$ such that 
$K_{X'}+B_{X'}+M_{X'}=p^*\omega$, where 
$B_{X'}$ is a subboundary $\mathbb R$-divisor 
such that $p_*B_{X'}$ is effective and 
$M_{X'}$ is a nef $\mathbb R$-divisor on $X'$. 
Let $\Exc(p)$ denote the exceptional locus of $p$. 
By taking some more blow-ups, we may further assume that 
$p(B^{=1}_{X'})=\Nqklt(X, \omega)$ and 
that $\Supp B_{X'}\cup \Supp p^{-1}_*B\cup \Exc(p)$ is contained 
in a simple normal crossing divisor $\Sigma$ on $X'$ (see 
Theorem \ref{u-thm2.12}). 

Let $\varepsilon$ be a small positive real number such that 
$(1-\varepsilon)(2+\alpha)>2$. 
We can take an effective $p$-exceptional 
$\mathbb Q$-divisor $E$ on $X'$ such that 
$-E$ is $p$-ample and 
that $M_{X'}+\varepsilon (p^*N-E)$ is semi-ample 
for any $\varepsilon >0$. 
For $0<\varepsilon \ll 1$, we 
put $\Delta_{\varepsilon}:=p_*(B_{X'}+\varepsilon E+G_{\varepsilon})$ where 
$G_\varepsilon$ is a general effective $\mathbb R$-divisor such that 
$G_{\varepsilon}\sim _{\mathbb R} M_{X'}+\varepsilon (p^*N-E)$, 
$\Supp G_{\varepsilon}$ and $\Supp \Sigma$ have no common 
irreducible components, $\lfloor G_\varepsilon \rfloor=0$, 
and $\Supp (\Sigma+G_{\varepsilon})$ is a simple 
normal crossing divisor on $X'$. 
Since the effective part of 
$-\lfloor B_{X'}+\varepsilon E+G_{\varepsilon}\rfloor$ 
is $p$-exceptional and $p(B^{=1}_{X'})=\Nqklt(X, \omega)$, 
we obtain 
\begin{equation*}
\begin{split}
\mathcal J(X, \Delta_{\varepsilon}) &=p_*\mathcal O_{X'} 
(-\lfloor B_{X'}+\varepsilon E+G_{\varepsilon}\rfloor) 
\\ &=p_*\mathcal O_{X'} 
(-\lfloor (B_{X'}+\varepsilon E+G_{\varepsilon})^{\geq 1}\rfloor) 
\\&=p_*\mathcal O_{X'}(-B^{=1}_{X'})
\\&=\mathcal I _{\Nqklt(X, \omega)}. 
\end{split}
\end{equation*}

We put 
$B_{\varepsilon}:=(1-\varepsilon)B$ and 
define 
$$
r_{\varepsilon}=\max\{t\geq 0\, |\, 
(X, \Delta_{\varepsilon}+tB_{\varepsilon}) 
\ \text{is log canonical at $P$}\}. 
$$
By construction, $\mult_P B_{\varepsilon}>2$ and 
$\Delta_{\varepsilon}$ is an effective $\mathbb R$-divisor 
on $X$. Therefore, we have $0<r_{\varepsilon}<1$. 
Note that $(X, \Delta_{\varepsilon})$ is klt at $P$. 
By construction again, 
there is an irreducible component $S_{\varepsilon}$ of $\Sigma$ such that 
$$
r_{\varepsilon}\mult _{S_{\varepsilon}}p^*B_{\varepsilon}+
\mult _{S_{\varepsilon}}B_{X'}+\varepsilon \mult _{S_{\varepsilon}}E=1. 
$$
Therefore, 
$$
0<r_{\varepsilon}=\frac{1-\mult _{S_{\varepsilon}}B_{X'}-
\varepsilon \mult _{S_{\varepsilon}} E}{
(1-\varepsilon)\mult _{S_{\varepsilon}}p^*B} <1
$$ 
holds. 
Since there are only finitely many components of $\Sigma$, 
we can take $\{\varepsilon_i\}_{i=1}^\infty$ and 
$\delta>0$ such that 
$0<\varepsilon _i\ll 1$, $\mathcal J(X, \Delta_{\varepsilon _i})
=\mathcal I_{\Nqklt(X, \omega)}$, 
$(X, \Delta_{\varepsilon_i})$ is klt at $P$, 
$(X, \Delta_{\varepsilon_i}+r_{\varepsilon_i}B_{\varepsilon_i})$ is log 
canonical at $P$ but is not klt at $P$ with $\delta<r_{\varepsilon_i}<1$ for 
every $i$, and $\varepsilon_i\searrow 0$ for 
$i\nearrow \infty$. 

By $p:X'\to X$, 
we get a natural quasi-log structure on $[X, \omega_{\varepsilon}]$ with 
$\omega_\varepsilon:=K_X+\Delta_{\varepsilon}+r_{\varepsilon} 
B_{\varepsilon}$ for any $\varepsilon=\varepsilon _i$ (see Example \ref{u-ex2.7}). 
Note that $[X, \omega_{\varepsilon}]$ is qlc in a 
neighborhood of $P$ since $(X, \Delta_{\varepsilon}+r_{\varepsilon}
B_{\varepsilon})$ is log canonical around $P$. 
Let $W_\varepsilon$ be the minimal qlc center of $[X, \omega_{\varepsilon}]$ 
passing through $P$, 
equivalently, let $W_\varepsilon$ be the minimal log canonical center 
of $(X, \Delta_{\varepsilon}+r_{\varepsilon}B_{\varepsilon})$ 
passing through $P$. 
Let $V_\varepsilon$ be the union of all qlc centers 
of $[X, \omega_{\varepsilon}]$ contained in $\Nqklt(X, 
\omega)=\Nklt(X, \Delta_{\varepsilon})$. 
We put $Z_\varepsilon =\Nqlc(X, \omega_{\varepsilon})
\cup V_\varepsilon\cup W_\varepsilon$ and 
$Y_\varepsilon=\Nqlc(X, \omega_{\varepsilon})\cup V_\varepsilon$. 
Then 
$[Z_\varepsilon, \omega_\varepsilon|_{Z_{\varepsilon}}]$ and 
$[Y_\varepsilon, \omega_\varepsilon|_{Y_{\varepsilon}}]$ have natural 
quasi-log structures induced from $[X, \omega_{\varepsilon}]$ 
by adjunction (see Theorem \ref{u-thm2.8} (i)). 
Since $$M-\omega_{\varepsilon}=
M-(K_X+\Delta_{\varepsilon}+r_{\varepsilon}B_{\varepsilon})
\sim _{\mathbb R} (1-r_{\varepsilon})(1-\varepsilon)N, $$ 
which is still ample, the restriction map 
\begin{equation}\label{u-eq3.1}
H^0(X, \mathcal O_X(M))\to H^0(Z_{\varepsilon}, 
\mathcal O_{Z_{\varepsilon}}(M))
\end{equation}
is surjective by Theorem \ref{u-thm2.8} (ii). 

\begin{case}\label{a-case1} 
If $\dim W_\varepsilon=0$, then 
$W_\varepsilon =P$ is isolated in $Z_\varepsilon$ by construction. 
Thus $Z_\varepsilon$ is the disjoint union of $P$ and 
$Y_\varepsilon$. 
Therefore, by \eqref{u-eq3.1}, 
the restriction map 
$$
H^0(X, \mathcal O_X(M))\to H^0(Y_{\varepsilon}, 
\mathcal O_{Y_\varepsilon}(M))\oplus H^0(P, \mathcal O_P(M))
$$ 
is surjective. 
This means that 
there exists $s\in H^0(X, \mathcal O_X(M))$ such that 
$s(P)\ne 0$ and 
$s\in H^0(X, \mathcal I_{Y_{\varepsilon}}\otimes 
\mathcal O_X(M))\subset 
H^0(X, \mathcal I_{\Nqklt(X, \omega)}\otimes \mathcal O_X(M))$. 
Note that $\mathcal I_{Y_\varepsilon}$ is the defining ideal sheaf of 
$Y_{\varepsilon}$ on $X$ and 
the natural inclusion $\mathcal I_{Y_{\varepsilon}}\subset 
\mathcal I_{\Nqklt(X, \omega)}$ holds by construction. 
This is what we wanted. 
\end{case}

\begin{case}\label{a-case2}
By Case \ref{a-case1}, 
we may assume that 
$\dim W_{\varepsilon}=1$ for 
any $\varepsilon=\varepsilon_i$. 
By construction, $P$ is not a qlc 
center of $[Z_\varepsilon, \omega_{\varepsilon}|_{Z_\varepsilon}]$. 
Therefore, $Z_\varepsilon$ is smooth at $P$ since $\dim W_\varepsilon
=1$ (see, for example, \cite[Theorem 6.3.11 (ii)]{fujino-foundations}). 
Let us consider 
$[Z_\varepsilon, \omega_\varepsilon|_{Z_\varepsilon}
+c_{\varepsilon}P]$ where $c_{\varepsilon}$ 
is the minimum positive real number such that 
$P$ is a qlc center of $[Z_\varepsilon, 
\omega_\varepsilon|_{Z_\varepsilon}
+c_{\varepsilon}P]$ (see Lemma \ref{u-lem2.13} and 
its proof). 
We write $\Delta_{\varepsilon}+r_{\varepsilon}B_{\varepsilon}
=W_{\varepsilon}+\Delta'_{\varepsilon}$. 
We put $\mult _P \Delta'_{\varepsilon}=\beta_{\varepsilon}\geq 0$. 
Then 
$$
\beta_{\varepsilon} =\mult _P \Delta_{\varepsilon} 
+r_{\varepsilon} (1-\varepsilon)(2+\alpha)-1\geq 
r_{\varepsilon} (1-\varepsilon)(2+\alpha)-1. 
$$
We note that 
$$
\beta_{\varepsilon}\leq \mult _P(\Delta'_{\varepsilon}|_{W_\varepsilon})<1
$$ 
holds because $(X, W_{\varepsilon} +\Delta'_{\varepsilon})$ is plt in 
a neighborhood of $P$. 
We note that $$(X, W_\varepsilon+\Delta'_{\varepsilon}+(1-\mult_P(\Delta'_
{\varepsilon}|_{W_{\varepsilon}}))H)$$ is log canonical 
but is not plt in a neighborhood of $P$, 
where $H$ is a general smooth curve passing through $P$. 
Therefore, 
$$
c_{\varepsilon}=1-\mult_P(\Delta'_{\varepsilon}|_{W_{\varepsilon}})
\leq 1-\beta_{\varepsilon}\leq 2-r_{\varepsilon}(1-\varepsilon)(2+\alpha). 
$$ 
In this situation, 
\begin{equation*}
\begin{split}
\deg ((M-\omega_{\varepsilon})|_{W_{\varepsilon}}-c_{\varepsilon}P)
&=(1-r_{\varepsilon})(1-\varepsilon)N\cdot W_\varepsilon -c_{\varepsilon}
\\&\geq 2(1-r_{\varepsilon})(1-\varepsilon)-2+
r_{\varepsilon}(1-\varepsilon)(2+\alpha) 
\\ &=(1-\varepsilon)r_{\varepsilon}\alpha-2\varepsilon. 
\end{split}
\end{equation*}
Here we used the assumption $N\cdot W_\varepsilon\geq 2$. 
We note that 
$(1-\varepsilon _i)r_{\varepsilon _i}\alpha -2\varepsilon _i>0$ for 
every $i\gg 0$ 
since $\varepsilon_i \searrow 0$ for $i\nearrow \infty$ and 
$r_{\varepsilon _i}>\delta>0$ for 
every $i$ by construction. 
Therefore, if we choose $0<\varepsilon=\varepsilon_i\ll 1$, then 
$$
\deg (M|_{W_{\varepsilon}}-(\omega_{\varepsilon}|_{W_{\varepsilon}}+c_
{\varepsilon}P))>0. 
$$ 
Thus, we see that the 
restriction map 
\begin{equation}\label{u-eq3.2} 
H^0(Z_{\varepsilon}, \mathcal O_{Z_{\varepsilon}}(M))
\to H^0(Y_{\varepsilon}, \mathcal O_{Y_{\varepsilon}}(M))\oplus 
H^0(P, \mathcal O_P(M))
\end{equation} 
is surjective by considering 
the quasi-log structure of $[Z_{\varepsilon}, 
\omega_{\varepsilon}|_{Z_{\varepsilon}}+c_{\varepsilon}P]$ 
with the aid of Theorem \ref{u-thm2.8}. 
By combining \eqref{u-eq3.2} with \eqref{u-eq3.1}, 
the restriction map 
$$
H^0(X, \mathcal O_X(M))\to 
H^0(Y_{\varepsilon}, \mathcal O_{Y_{\varepsilon}}(M))\oplus 
H^0(P, \mathcal O_P(M))
$$ is surjective. 
As in Case \ref{a-case1}, 
we get $s\in H^0(X, \mathcal O_X(M))$ such that 
$s(P)\ne 0$ and 
$s\in H^0(X, \mathcal I_{\Nqklt(X, \omega)}\otimes 
\mathcal O_X(M))$. 
\end{case}

Anyway, we can construct $s\in H^0(X, \mathcal I_{\Nqklt(X, \omega)}\otimes 
\mathcal O_X(M))$ such that $s(P)\ne 0$ when $P$ is a smooth point 
of $X$. 

\end{step}
\begin{step}\label{a-step2}
In this step, we assume that $P$ is a singular 
point of $X$. 
Let $\pi:Y\to X$ be the minimal 
resolution of $P$. 
Then we have the following 
commutative 
diagram 
$$
\xymatrix{
X' \ar[dr]^-p\ar[d]_-q& \\ 
Y \ar[r]_-\pi& X, 
} 
$$ 
where $p:X'\to X$ is a projective 
birational morphism from a smooth surface $X'$ constructed in 
Step \ref{a-step1} by using Theorem \ref{u-thm2.12}. 
Let $\Delta_{\varepsilon}$ be an effective $\mathbb R$-divisor on $X$ 
as in Step \ref{a-step1}. 
We put $\pi^*(K_X+\Delta_{\varepsilon})=K_Y+\Delta^Y_{\varepsilon}$. 
We note that $\Delta^Y_{\varepsilon}$ is effective since 
$\pi$ is the minimal resolution of $P$. 
By construction, $\pi$ is an isomorphism outside $\pi^{-1}(P)$. 
In particular, $\pi$ is an isomorphism 
over some open neighborhood of $\Nqklt(X, \omega)$. 
Therefore, $\mathcal J(Y, \Delta^Y_{\varepsilon})=\mathcal I_{\pi^{-1}
(\Nqklt(X, \omega))}$ holds since 
$\mathcal J(X, \Delta_{\varepsilon})=\mathcal I_{\Nqklt(X, \omega)}$, where 
$\mathcal I_{\pi^{-1}(\Nqklt(X, \omega))}$ is the defining ideal sheaf of 
$\pi^{-1}(\Nqklt(X, \omega))$. 
Since $(\pi^*N)^2>4$, we can take an effective $\mathbb R$-divisor 
$B$ on $X$ such that $B\sim _{\mathbb R}N$ and $\mult_Q D>2$, where 
$D=\pi^*B$, for some $Q\in \pi^{-1}(P)$. 
We put $D_{\varepsilon}:=(1-\varepsilon)D$ and 
$B_{\varepsilon}:=(1-\varepsilon)B$ and 
define 
$$
s_\varepsilon=\max\{t\geq 0\, |\, 
(Y, \Delta^Y_{\varepsilon}+tD_{\varepsilon}) 
\ \text{is log canonical at any point of $\pi^{-1}(P)$}\}.
$$
Then we have $0<s_\varepsilon<1$ 
since $\mult_QD_{\varepsilon}>2$ for $0<\varepsilon 
\ll 1$. 
Therefore, we can take $Q_{\varepsilon}\in \pi^{-1}(P)$ such that 
$(Y, \Delta^Y_{\varepsilon}+s_{\varepsilon}D_{\varepsilon})$ is log canonical 
but is not klt at $Q_{\varepsilon}$. 
As in Step \ref{a-step1}, 
we may assume that $\Supp B_{X'}\cup \Supp p^{-1}_*B\cup 
\Exc(p)$ is contained in a simple normal crossing divisor $\Sigma$ on $X'$. 
By the same argument as in Step \ref{a-step1}, 
we can take some point $R$ on $\pi^{-1}(P)$, 
$\{\varepsilon_i\}_{i=1}^\infty$, and $\delta>0$ such that 
$0<\varepsilon _i \ll 1$,
$\varepsilon _i\searrow 0$ for $i\nearrow \infty$, 
$\mathcal J(Y, \Delta^Y_{\varepsilon_i})
=\mathcal I_{\pi^{-1}(\Nqklt(X, \omega))}$, 
$(Y, \Delta^Y_{\varepsilon_i}+s_{\varepsilon _i}D_{\varepsilon_i})$ is log 
canonical at $R$ but is not klt at 
$R$ with $\delta<s_{\varepsilon_i}<1$ for every $i$ since there are 
only finitely many components of $\Sigma$. 
By $q:X'\to Y$, we have 
a natural quasi-log structure on $[Y, \omega^Y_\varepsilon]$ with 
$\omega^Y_{\varepsilon}:=K_Y+\Delta^Y_{\varepsilon}+
s_{\varepsilon}D_\varepsilon$ for any $\varepsilon =\varepsilon_i$ 
(see Example \ref{u-ex2.7}). 
If there is a one-dimensional 
qlc center $C$ of $[Y, \omega^Y_{\varepsilon}]$ for some 
$\varepsilon$ with 
$(\pi^*M-\omega^Y_{\varepsilon})\cdot C=0$, then $C\subset \pi^{-1}(P)$. 
This is because 
$$(\pi^*M-\omega^Y_{\varepsilon})\cdot C
=(1-s_{\varepsilon})(1-\varepsilon)N\cdot \pi_*C=0. $$ 
This means that $P$ is a qlc center of $[X, \omega_{\varepsilon}]$, where 
$\omega_{\varepsilon}:=K_X+\Delta_{\varepsilon}+s_{\varepsilon}B_\varepsilon$. 
In this case, we can use Case \ref{a-case1} in Step \ref{a-step1}. 
Therefore, for any $\varepsilon=\varepsilon_i$, 
we may assume that $(\pi^*M-\omega^Y_{\varepsilon})\cdot 
C>0$ for every one-dimensional qlc 
center $C$ of $[Y, \omega^Y_{\varepsilon}]$. 
Now we can apply the arguments for $[X, \omega_{\varepsilon}]$ and 
$M$ in Step \ref{a-step1} to $[Y, \omega^Y_{\varepsilon}]$ and 
$\pi^*M$ here. 
We note that $\pi^*M-\omega^Y_{\varepsilon}$ is not ample but is 
nef and log big with respect to $[Y, \omega^Y_{\varepsilon}]$. 
Thus we can use Theorem \ref{u-thm2.8} (ii). 
Then we obtain $s^Y\in H^0(Y, \mathcal I_{\pi^{-1}(\Nqklt(X, \omega))}
\otimes \mathcal O_Y(\pi^*M))$ such that 
$s^Y(R)\ne 0$. Therefore, we have  
$s\in H^0(X, \mathcal I_{\Nqklt(X, \omega)}\otimes 
\mathcal O_X(M))$ such that 
$\pi^*s=s^Y$. 
In particular, $s(P)\ne 0$. 
This is what we wanted. 
\end{step}
Anyway, we finish the proof of Theorem \ref{a-thm3.2}. 
\end{proof}

Now the proof of Corollary \ref{a-cor1.5} is easy. 

\begin{cor}[Corollary \ref{a-cor1.5}]\label{a-cor3.3}
Conjecture \ref{a-conj1.2} is true in dimension two.
\end{cor}

\begin{proof} 
Let $P$ be an arbitrary closed point of $X$ and let 
$W$ be the unique minimal qlc stratum of $[X, \omega]$ 
passing through $P$. Note that $W$ is irreducible by definition.
By adjunction (see Theorem \ref{u-thm2.8} (i)), 
$[W, \omega|_W]$ is an irreducible quasi-log canonical pair. 
By Theorem \ref{u-thm2.8} (ii), the natural restriction map 
\begin{equation}\label{a-eq3.2}
H^0(X, \mathcal O_X(M))\to H^0(W, \mathcal O_W(M))
\end{equation}
is surjective. From now on, we will see that $|M|$ is 
basepoint-free in a neighborhood of $P$. 
If $W=P$, that is, $P$ is 
a qlc center of $[X, \omega]$, then the 
complete linear system $|M|$ is obviously basepoint-free in 
a neighborhood of $P$ by the surjection \eqref{a-eq3.2}. 
Let us consider the case where $\dim W=1$. 
We put $M'=M|_W$ and $N'=N|_W=M'-\omega|_W$.
Then $\deg N'=N\cdot W >1$ by assumption.
Therefore, by Theorem \ref{a-thm3.1}, 
$|M'|$ is basepoint-free at $P$ because $[W, \omega|_W]$ is 
an irreducible projective quasi-log canonical curve.
Therefore, by the surjection \eqref{a-eq3.2}, 
we see that $|M|$ is basepoint-free in a neighborhood of $P$. 
Thus we may assume that 
$\dim W=\dim X=2$ and $X$ is irreducible by replacing $X$ with 
$W$ since the restriction map \eqref{a-eq3.2} 
is surjective. 
Therefore, 
we can assume that 
$X$ is irreducible and that $X$ is the unique qlc stratum of $[X, \omega]$ passing 
through $P$. In particular, $X$ is normal at $P$ 
(see, for example, \cite[Theorem 6.3.11 (ii)]{fujino-foundations}).
Let $\nu: \widetilde{X}\to X$ be the normalization. 
Note that $[\widetilde{X}, \nu^*\omega]$ is 
a qlc pair by Theorem \ref{u-thm2.11}.
We put $\widetilde{M}=\nu^*M$ 
and $\widetilde{N}=\nu^*N=\widetilde{M}-\nu^*\omega$.
It is obvious that $\widetilde{M}$ is Cartier.
Moreover, we have $(\widetilde{N})^2=N^2> 4$
and $\widetilde{N}\cdot Z\geq N \cdot \nu(Z)\geq 2$ for every curve 
$Z$ on $\widetilde{X}$. Note that $\dim \nu(Z)=\dim Z=1$ 
since $\nu$ is finite.
We also note that $P':=\nu^{-1}(P)$ is a point 
since $\nu: \widetilde{X}\to X$ is an isomorphism over some open neighborhood 
of $P$. This is because $X$ is normal at $P$. 
Now the assumptions of Theorem \ref{a-thm3.2} are all satisfied. 
Therefore, there is a section $s' \in H^0(\widetilde {X}, 
\mathcal I_{\Nqklt(\widetilde{X}, \nu^*\omega)}
\otimes \mathcal  O_{\widetilde{X}}(\widetilde{M}))$ 
such that $s'(P')\neq 0$.
We note that the non-normal part of $X$ is contained in $\Nqklt(X, \omega)$ 
(see, for example, \cite[Theorem 6.3.11 (ii)]{fujino-foundations}) 
and that the equality 
$$
\nu_*\mathcal I_{\Nqklt(\widetilde{X}, \nu^*\omega)}=
\mathcal I_{\Nqklt(X, \omega)}
$$ holds 
by Theorem \ref{u-thm2.11}. 
Therefore, we have 
$$
H^0(\widetilde{X}, \mathcal I_{\Nqklt(\widetilde{X}, \nu^*\omega)
}\otimes \mathcal O_{\widetilde{X}}(\widetilde{M})) \simeq  
H^0(X,\mathcal I_{\Nqklt(X, \omega)} \otimes 
\mathcal O_{X}(M)).
$$ 
Thus we can descend the section $s'$ on $\widetilde X$ to 
a section $s\in H^0(X, \mathcal I_{\Nqklt(X, \omega)}
\otimes \mathcal O_X(M))$ with $s(P)\neq 0$. 
Therefore, by this section $s \in H^0(X, \mathcal O_{X}(M))$,  
we see that $|M|$ is basepoint-free in a neighborhood of $P$. 
This is what we wanted. 
\end{proof}

We close this section with the proof of Corollary \ref{a-cor1.6}. 

\begin{proof}[Proof of Corollary \ref{a-cor1.6}]
Let $(X, \Delta)$ be a projective 
semi-log canonical 
pair. 
Then, by Theorem \ref{u-thm2.10}, 
$[X, K_X+\Delta]$ is a quasi-log canonical 
pair. 
Therefore, Corollary \ref{a-cor1.6} is a direct consequence of 
Theorem \ref{a-thm1.3} and Corollary \ref{a-cor1.5}. 
\end{proof}

%%%%%%%%%%%%%%%

\end{document}